\documentclass{article}

\usepackage{amsmath,amssymb,amsthm,amsfonts}
\usepackage{url}
\theoremstyle{plain}
\newtheorem{thm}{Theorem}[section]

\newtheorem{prop}[thm]{Proposition}
\newtheorem*{cor}{Corollary}

\theoremstyle{definition}
\newtheorem{defn}{Definition}[section]

\newtheorem{rem}{Remark}[section]

\begin{document}

\title{On abelian and additive complexity in infinite words}

\author{Hayri Ardal, Tom Brown, Veselin Jungi\'{c}, Julian Sahasrabudhe }
\date{July 2011}
\maketitle

\begin{abstract}  

The study of the structure of infinite words having \textit{bounded abelian complexity} was initiated by G. Richomme, K. Saari, and L. Q. Zamboni \cite{RSZ}.  In this note we define \textit{bounded additive complexity} for infinite words over a finite subset of $\mathbb{Z}^m.$   We provide an alternative proof of one of the results of \cite{RSZ}.

\end{abstract}


\section{Introduction}

Recently the study of infinite words with \textit{bounded abelian complexity} was initiated by G. Richomme, K. Saari, and L. Q. Zamboni \cite{RSZ}.  (See also \cite{CRSZ} and the references in \cite{CRSZ} and \cite{RSZ}.)  In particular, it is shown (in \cite{RSZ}) that if $\omega$ is an infinite word with bounded abelian complexity, then $\omega$ has \textit{abelian $k$-factors} for all $k \ge 1.$   (All these terms are defined below.)\

In this note we define \textit{bounded additive complexity}, and we show in particular that if $\omega$ is an infinite word (whose alphabet is a finite subset $S$ of $\mathbb{Z}^m$ for some $m \ge 1$) with bounded additive complexity, then $\omega$ has \textit{additive $k$-factors} for all $k \ge 1.$  As we shall see, this provides an alternative proof of the just-mentioned result concerning abelian $k$-factors.\

We are motivated by the following question.  In \cite{Freedman}, \cite{Jaroslaw}, \cite {HH}, and \cite{Pirillo}, it is asked whether or not there exists an infinite word on a finite subset of $\mathbb{Z}$ in which there do not exist two adjacent factors with equal lengths and equal sums.  (The \textit{sum} of the factor $x_1x_2\dots x_n$ is $x_1+x_2+\dots + x_n.$)  This question remains open, although some partial results can be found in \cite{Au}, \cite{CCSS}, \cite{Freedman}.


\section{Additive complexity}

\subsection{Infinite words on finite subsets of $\mathbb{Z}$}


\begin{defn}
  Let $\omega$ be an infinite word on a finite subset $S$ of $\mathbb{Z}$.  For a factor $B=x_1x_2\dots x_n$ of $\omega$, $\sum B$ denotes the sum $x_1+x_2+\cdots +x_n.$  Let $$\phi_\omega(n) = \{ \sum  B :  \mbox{ $B$ is a factor of $\omega$ with length $n$}\}.$$ The function $|\phi_{\omega}|$ (where $|\phi_{\omega}|(n) = |\phi_{\omega}(n)|, n \ge 1)$ is called the \textit{additive complexity} of the word $\omega.$\
  
  If $B_1B_2\cdots B_k$ is a factor of $\omega$ such that $|B_1| = |B_2|=\cdots = |B_k|$ and $\sum B_1=\sum B_2=\cdots=\sum B_k,$ we call $B_1B_2\cdots B_k$ an \textit{additive $k$-power}.

  We say that $\omega$ has \textit{bounded additive complexity} if any one (and hence all) of the three conditions in the following proposition (Proposition 2.1) hold.\end{defn}


\begin{prop}  

Let $\omega$ be an infinite word on the alphabet $S$, where $S$ is a finite subset of $\mathbb{Z}$.  Then the following three statements are equivalent.\

1.  There exists $M_1$ such that if $B_1B_2$ is a factor of $\omega$ with $|B_1| = |B_2|,$ then $|\sum B_1 - \sum B_2| \le M_1.$\

2. There exists $M_2$ such that if $B_1,B_2$ are factors of $\omega$ (not necessarily adjacent) with $|B_1| = |B_2|,$ then $|\sum B_1 - \sum B_2| \le M_2.$\

3.  There exists $M_3$ such that $|\phi_{\omega}(n)| \le M_3$ for all $n \ge 1$.

\end{prop}

\begin{proof}

We will show that $1\Leftrightarrow 2$ and $2\Leftrightarrow 3.$\

  Clearly $2 \Rightarrow 1$.  Now assume that $1$ holds, that is, if $B_1B_2$ is any factor of $\omega$ with $|B_1| = |B_2|,$ it is the case that $|\sum B_1 - \sum B_2| \le M_1.$ Now let $B_1$ and $B_2$ be factors of $\omega$ with $|B_1| = |B_2|,$ and assume that $B_1$ and $B_2$ are non-adjacent, with $B_1$ to the left of $B_2$.  

Thus, assume that $$B_1A_1A_2B_2$$

 is a factor of $\omega$, where 
 
 $$|A_1| = |A_2|\  or\  |A_1| = |A_2|+1.$$  
 
 Let $$C_1 = B_1A_1, C_2 = A_2B_2.$$  Then $$|C_1| = |C_2|\ or \ |C_1| = |C_2|+1.$$

Now $$\sum C_1 - \sum C_2 = (\sum B_1 + \sum A_1) - (\sum A_2 + \sum B_2),$$ 

or $$\sum B_1 -\sum B_2 = (\sum C_1 - \sum C_2) + (\sum A_2 - \sum A_1).$$

Therefore, since $A_1, A_2$ and $C_1, C_2$ are adjacent, we have $$ |\sum A_2 - \sum A_1| \le M_1+\max S, \ \  |\sum C_1 - \sum C_2|\le M_1+\max S,$$

and  $$|\sum B_1 -\sum B_2| \le 2M_1 + 2\max S,$$ 
so that we can take $M_2 = 2M_1 + 2\max S.$  Thus $1 \Rightarrow 2.$\\

Next we show that $2\Rightarrow 3.$  Thus we assume there exists $M_2$ such that whenever $B_1,B_2$ are factors of $\omega$ (not necessarily adjacent) with $|B_1| = |B_2|,$ it is the case that $|\sum B_1 - \sum B_2| \le M_2.$\

Let $n$ be given, and let $\sum B_1 = \min \phi_\omega(n).$  Then for any $B_2$ with $|B_2| = n,$ we have $\sum B_2 = \sum B_1 + (\sum B_2 - \sum B_1).$  Therefore $\sum B_2 \le \sum B_1 + M_2.$  This means that $\phi_\omega(n) \subset [\sum B_1, \sum B_1 + M_2],$ so that $|\phi_{\omega}(n)| \le M_2 + 1.$\\

Finally, we show that $3\Rightarrow 2.$  We assume there exists $M_3$ such that $|\phi_{\omega}(n)| \le M_3$ for all $n \ge 1$.  Suppose that  $B_1$ and $B_2$ are factors of $\omega$  such that $|B_1|=|B_2|= n$ and $\sum B_1 = \min \phi_{\omega}(n)$, $\sum B_2=\max \phi_{\omega}(n).$  To simplify the notation, for all $a \le b$ let $\omega[a,b]$ denote $x_ax_{a+1}\dots x_b$, and let us assume that $B_1 = \omega[1, n], B_2 = \omega[q+1, q+n],$ where $q>1.$ \  
  
  For each $i, 0 \le i \le q,$ let $b_i$ denote the factor $\omega[i+1, i+n].$  Thus $B_1 = b_0, B_2 = b_{q},$ and the factor $b_{i+1}$ is obtained by shifting $b_i$ one position to the right. Clearly $$\sum b_{i+1} - \sum b_i \le \max S - \min S.$$\
  
  Since $|b_0|= |b_1|= \cdots = |b_{q}| = n,$ and $|\phi_{\omega}(n)| \le M_3,$   there can be at most $M_3$ distinct numbers in the sequence $\sum B_1 = \sum b_0, \sum b_1, \dots, \sum b_{q}=\sum B_2.$  Let these numbers be $$\sum B_1=c_1 < c_2 < \cdots < c_r = \sum B_2,$$ where $r \le M_3.$\  
  
  Since $\sum b_{i+1} - \sum b_i \le \max S - \min S, \ 0 \le i \le q,$ it follows that $c_{j+1} - c_j \le \max S - \min S ,\ 0 \le i \le r-1,$ and hence that $$|\sum B_1 - \sum B_2| \le (M_3-1)(\max S - \min S).$$
  
  \end{proof}
  
  
  \begin{thm}

Let $\omega$ be an infinite word on  a finite subset of $\mathbb{Z}$. Assume that $\omega$ has bounded additive complexity. Then $\omega $ contains an additive $k$-power for every positive integer $k$.  

\end{thm}


  \begin{proof}
  
  Let $\omega = x_1x_2x_3 \cdots$ be an infinite word on the finite subset $S$ of $\mathbb{Z}$, and assume that whenever $B_1,B_2$ are factors of $\omega$ (not necessarily adjacent) with $|B_1| = |B_2|,$ then $|\sum B_1 - \sum B_2| \le M_2.$\  (This is from part 2 of Proposition 2.1.)

  Define the function $f$ from $\mathbb{N}$ to $\{0, 1, 2, \dots, M_2\}$ by $$f(n)=x_1 + x_2 + x_3 + \cdots + x_n  \  \pmod{M_2 + 1}, \ \ n \ge 1 .$$\
  
  This is a finite coloring of $\mathbb{N}$; by van der Waerden's theorem, for any $k$ there are $t,s$ such that $$f(t)=f(t+s)=f(t+2s)=\cdots f(t+ks).$$
  
  Setting $$ B_i= \omega[t+(i-1)s+1, t+is], \ \ 1 \le i \le k,$$
  we have $$\sum B_1 \equiv \sum B_2 \equiv \cdots \equiv \sum B_k \pmod{M_2 + 1}.$$  \
  
  Since $B_1B_2\cdots B_k$ is a factor of $\omega$ with $|B_i| = |B_j|, 1 \le i <j \le k,$ we have $|\sum B_i - \sum B_j| \le M_2$ and $\sum B_i \equiv \sum B_j \pmod{M_2 + 1},$ hence $\sum B_i = \sum B_j.$  \

  Thus $|B_1| = |B_2|=\cdots = |B_k|$ and $\sum B_1=\sum B_2=\cdots=\sum B_k,$ and $\omega$ contains the additive $k$-power $B_1B_2\cdots B_k$.
  
  \end{proof}\
  
  
  \subsection{Infinite words on  subsets of $\mathbb{Z}^m$}

  Let us use the notation $(u)_j$ for the $jth$ coordinate of $u \in \mathbb{Z}^m.$  That is, if $u = (u_1, \dots, u_m)$ then $(u)_j = u_j.$  Also, $|u| = |(u_1, \dots, u_m)|$ denotes the vector $(|u_1|, \dots, |u_m)|).$  In other words, $(|u|)_j = |(u)_j|.$\
  
  For factors $B_1, B_2$ of an infinite word $\omega$ on a finite subset $S$ of $\mathbb{Z}^m$, the notation $|\sum B_1 - \sum B_2| \le M_1$ means that $(|\sum B_1 - \sum B_2|)_j \le M_1,\  1 \le j \le m.$\

  Now we suppose that $\omega$ is an infinite word on a finite subset $S$ of $\mathbb{Z}^m$ for some $m \ge 1.$  The definition of $\phi_\omega$ and the additive complexity of $\omega$ is exactly as in Definition 1.1 above.  The function $$\phi_\omega(n) = \{ \sum  B :  \mbox{ $B$ is a factor of $\omega$ with length $n$}\}$$ is called the \textit{additive complexity} of the word $\omega.$\\
  
  By working with the coordinates $(B_1)_j, (|\sum B_1 - \sum B_2|)_j,$ we easily obtain the following results.   
  
    
  \begin{prop}  
  
  Proposition 2.1  remains true when $\mathbb{Z}$ is replaced by $\mathbb{Z}^m$.
  
  \end{prop}
  
  
 \begin{thm}Let $\omega$ be an infinite word on  a finite subset of $\mathbb{Z}^m$ for some $m \ge 1$. Assume that $\omega$ has bounded additive complexity. Then $\omega$ contains an additive $k$-power for every positive integer $k$.   
 \end{thm}
 
 The following is a re-statement of Theorem 2.4, in terms of $m$ infinite words on $\mathbb{Z},$ rather than one infinite word on $\mathbb{Z}^m.$
 
 
\begin{thm}
Let $m\in \mathbb{N}$ be given, and let $S_1, S_2, \dots, S_m$ be finite subsets of $\mathbb{Z}.$  Let $\omega_j$ be an infinite word on $S_j$ with bounded additive complexity, $1 \le j \le m.$  Then for all $k \ge 1$, there exists a $k$-term arithmetic progression in $\mathbb{N},t, t+s, t+2s, \dots, t+ks$ such that for all $j, 1 \le j \le m,$
$$\sum \omega_j[t+1, t+s] = \sum \omega_j[t+s+1, t+2s] = \cdots = \sum \omega_j[t+(k-1)s+1, t+ks].$$\
Thus $\omega_1, \omega_2, \cdots, \omega_m$ have ``simultaneous" additive $k$-powers for all $k \ge 1.$
\end{thm}

  
  \section{Abelian complexity}
  
  
  \begin{defn}
    Let $\omega$ be an infinite word on a finite alphabet.  Two factors of $\omega$  are called \textit{abelian equivalent} if one is a permutation of the other.  If the alphabet is $A = \{a_1,a_2,\dots, a_t\},$ and the finite word $B$ is a factor of $\omega,$ we write $\psi(B) = (u_1, u_2, \dots, u_t),$ where $u_i$ is the number of occurrences of the letter $i$ in the word $B, 1 \le i \le t.$  We call $\psi(B)$ the \textit{Parikh vector} associated with $B$.\
    
    Let $\psi_{\omega}(n) = \{\psi(B): \mbox{$B$ is a factor of $\omega, |B| = n$}\}$.  The function $\rho_{\omega}^{ab},$ defined by $\rho_{\omega}^{ab}(n) = |\psi_{\omega}(n)|, n \ge 1,$ is called the \textit{abelian complexity} of $\omega$.\
    
    Thus $\rho_{\omega}^{ab}(n)$ is the largest number of factors of $\omega$ of length $n$, no two of which are abelian equivalent. If there exists $M$ such that $\rho_{\omega}^{ab}(n) \le  M$ for all $n \ge 1,$ then $\omega$ is said to have \textit{bounded abelian complexity}.     \

The word $B_1B_2\cdots B_k$ is called an \textit{abelian $k$-power} if $B_1,$ $B_2,$ $\dots, B_k$  are pairwise abelian equivalent.  (Being abelian equivalent, they all have the same length.)

    Recall that we are using the notation $|(u_1,u_2, \dots, u_t)| \le M$ to denote $|u_i| \le M, 1 \le i \le t.$
\end{defn}


\begin{prop}  Let $\omega$ be an infinite word on a $t$-element alphabet $S$.  Then the following three statements are equivalent.\

1.  There exists $M_1$ such that if $B_1B_2$ is a factor of $\omega$ with $|B_1| = |B_2|,$ then $|\psi(B_1) - \psi(B_2)| \le M_1.$\

2. There exists $M_2$ such that if $B_1,B_2$ are factors of $\omega$ (not necessarily adjacent) with $|B_1| = |B_2|,$ then $|\psi(B_1) - \psi(B_2)| \le M_2.$\

3.  There exists $M_3$ such that such that $\rho_{\omega}^{ab}(n) \le M_3$ for all $n \ge 1$.

\end{prop}

\begin{proof}  We show that $1\Leftrightarrow 2$ and $2\Leftrightarrow 3.$\

  Clearly $2 \Rightarrow 1$.  Now assume that $1$ holds, that is, if $B_1B_2$ is any factor of $\omega$ with $|B_1| = |B_2|,$ it is the case that $|\psi(B_1) - \psi(B_2)| \le M_1.$ Now let $B_1$ and $B_2$ be factors of $\omega$ with $|B_1| = |B_2|,$ and assume that $B_1$ and $B_2$ are non-adjacent, with $B_1$ to the left of $B_2$.  

Thus, assume that $$B_1A_1A_2B_2$$

 is a factor of $\omega$, where 
 
 $$|A_1| = |A_2|\  or\  |A_1| = |A_2|+1.$$ 
 Now we proceed exactly as in the proof of  $1 \Rightarrow 2$ in Proposition 2.1, noting that $|\psi(A_1) - \psi(A_2)| \le M_1 + 1.$\\
 
 Next we show that $2\Rightarrow 3.$  Thus we assume there exists $M_2$ such that whenever $B_1,B_2$ are factors of $\omega$ (not necessarily adjacent) with $|B_1| = |B_2|,$ it is the case that $|\psi(B_1) - \psi(B_2)| \le M_2.$\

Let $n$ be given, and let $B_1 \in \psi_{\omega}(n).$  Then for any $B_2$ with $|B_2| = n,$ we have $\psi(B_2) = \psi(B_1) + (\psi(B_2) - \psi(B_1)).$  Therefore $|\psi(B_2)| \le |\psi(B_1)| + M_2.$  (This inequality is component-wise, that is,  $(|\psi(B_2)|)_j \le (|\psi(B_1)|)_j + M_2, 1 \le j \le t.$)\

Therefore there are at most $2M_2-1$ choices for each component of $B_2$, and hence $\rho_{\omega}^{ab}(n) \le (2M_2-1)^t.$\\

Finally, we show that $3\Rightarrow 2.$  We assume there exists $M_3$ such that $\rho_{\omega}^{ab}(n) \le M_3$ for all $n \ge 1$.  \

Since $|\psi(xB) - \psi(By)| \le 1$ for all $x,y \in S$, it follows that if $\omega$ has factors $B_1, B_2$  of length $n$ where for some $j, 1 \le j \le t, (\psi(B_1))_j = p$ and $(\psi(B_2))_j = p+q$, then $\omega$ has factors $C_r$ of length $n$ with $(\psi(C_r))_j = p+r, 0 \le r \le q.$  (This is discussed in more detail in \cite{RSZ}.)  Thus $|\psi(B_1) - \psi(B_2)| \ge M_3$ implies $\rho_{\omega}^{ab}(n) \ge M_3 + 1.$  Since we are assuming $\rho_{\omega}^{ab}(n) \le M_3, n \ge 1,$ we conclude that $|\psi(B_1) - \psi(B_2)| \le M_3-1$ whenever $|B_1| = |B_2|$.  Hence $|\psi(B_1) - \psi(B_2)| \le M_3 - 1$ whenever $|B_1| = |B_2|.$

\end{proof}


    \begin{rem}  
    
  To see that bounded sum complexity is indeed weaker than bounded abelian complexity, consider the following example. Let $\sigma = x_1x_2x_3 \cdots$ be the binary sequence constructed by Dekking [2] which has no abelian 4th power.  In $\sigma,$ replace every 1 by 12, and replace every 0 by 03, obtaining the sequence $\tau.$  If $\tau$ had an abelian 4th power $ABCD$, then the number of 2s in each of $A,B,C,D$ are equal, and similarly for the number of 3s.  But then dropping the 2s and 3s from $ABCD$ would give an abelian 4th power in $\sigma$, a contradiction.  Hence $\tau$ does not have bounded abelian complexity.  Now let a factor $B$ of $\tau$ be given.  By shifting $B$ to the right or left, we see, by examining cases, that if $|B|$ is even then $\sum B = \frac{3}{2}|B| + s,$ where $s \in \{-1,0,1\}.$  If $|B|$ is odd, then $\sum B = \frac{3}{2}(|B|-1)+s,$ where $s \in \{0,1,2,3\}.$  Hence $|\phi_{\tau}(n)| \le 4$ for all $n \ge 1,$ and $\tau$ does have bounded sum complexity.

  \end{rem}
  
  
    \begin{defn}
  
  Let $S=\{a_1, a_2, \dots, a_m\}$ be a subset of $\mathbb{Z}$, and let $\omega = x_1x_2x_3\cdots$ be an infinite word on the alphabet S.  For each $j, 1  \le j \le m,$ let $a_j'$ be the element of $\mathbb{Z}^m$ which has $a_j$ in the in the $jth$ coordinate and $0's$ elsewhere.  Let $\omega' = x_1'x_2'x_3'\cdots$ be the word on the subset $S'$ of $\mathbb{Z}^m, S'  = \{a_1', a_2', \dots, a_m'\},$ obtained from $\omega$ by replacing each $a_j$ by $a_j', \ 1  \le j \le m.$  It is convenient to visualize each $a_j'$ as a column vector, rather than as a row vector.\\
  
  \end{defn}
  
  
  \begin{thm}

  Referring to Definition 2.2, consider the following statements concerning $\omega$ and $\omega'$:\\  
  
  1.   $\omega$ has bounded abelian complexity.\  
  
  2.   $\omega'$ has bounded abelian complexity.\
  
  3.   $\omega'$ has bounded additive complexity.\
  
  4.   $\omega'$ contains an additive $k$-power for all $k \ge 1.$\

  5.   $\omega'$ contains an abelian $k$-power or all $k \ge 1$, \
  
  6.   $\omega$ contains an abelian $k$-power for all $k \ge 1$ \\
    
  Then $1 \Leftrightarrow 2 \Leftrightarrow 3$, $4 \Leftrightarrow 5 \Leftrightarrow 6,$ $3 \Rightarrow 4,$ and  $4 \nRightarrow 3$  \\
  
  \end{thm}
  
  \begin{proof}
  
  Clearly $1 \Leftrightarrow 2$ and $5 \Leftrightarrow 6.$\\
  
    The linear independence of $S'$ over $\mathbb{Z}$ implies that $2 \Leftrightarrow 3$ and $4 \Leftrightarrow 5.$\\

 The implication $3 \Rightarrow 4$ is a special case of the second part of Theorem 2.4.\\
 
   To see that $4 \nRightarrow 3$,  note that if $4 \Rightarrow 3$ then $6 \Rightarrow 1$, which is shown to be false by the Champernowne word \cite{Champ}
  $$C = 01101110010111011110001001 \cdots,$$
obtained by concatenating the binary representations of $0, 1, 2, \dots$ .  This word has arbitrarily long strings of 1's (and 0's), hence satisfies condition 6; but $C$ does not satisfy condition 1.  (Clearly for the sequence $C$, $\rho_{C}^{ab}(n) = n+1$ for all $n \ge 1.$)\\

  \end{proof}
  

  \begin{cor}
  
  Every infinite word with bounded abelian complexity has an abelian $k$-power for every $k$.
  
  \end{cor}
  
  \section{A more general statement}
  
  One can cast the arguments above into a more general form, and prove (we leave the details to the reader) the following statement.
  
  
  \begin{thm}
  Let $S$ be a finite set, and let $S^+$ denote the free semigroup on $S$.  For $t \in \mathbb{N}$, let $$\mu:S^+ \rightarrow \mathbb{Z}^t$$ be a morphism, that is, for all $B_1, B_2 \in S^+,$  $$\mu(B_1B_2) = \mu(B_1) + \mu(B_2).$$  Let $\omega$ be an infinite word on $S$.  Assume further that there exists $M \in \mathbb{N}$ such that $$|B_1| = |B_2| \Rightarrow ||\mu(B_1) - \mu(B_2)\\\\|| \le M,$$ where $||\cdot ||$ denotes Euclidean distance in $\mathbb{Z}^t.$  Then for all $k \ge 1,$   $\omega$ contains a $k$-power modulo $\mu,$ that is, $\omega$ has a factor $B_1B_2 \cdots B_k$ with $$|B_1| = |B_2| = \cdots = |B_k|,\ \  \mu(B_1) = \mu(B_2) = \cdots = \mu(B_k).$$
  
  \end{thm}
  
  Thus taking $S$ to be a finite subset of $\mathbb{Z}^m,$ and $\mu(B) = \sum B \in \mathbb{Z}^m,$ we obtain Theorem 2.4.\
  
  Taking $S$ to be a finite set and $\mu(B) = \psi(B) \in \mathbb{Z}^{|S|},$ we obtain the Corollary to Theorem 3.2

\bigskip

\noindent\textit{Department of Mathematics,
Simon Fraser University, Burnaby, BC, Canada, V5A 1S6\\
hardal@sfu.ca\\
tbrown@sfu.ca\\
vjungic@sfu.ca\\
jds16@sfu.ca\\}

\end{document}